\documentclass[a4paper,12pt,reqno]{amsart}

\usepackage[utf8]{inputenc}
\usepackage[T1]{fontenc}
\usepackage{lmodern}
\usepackage[english]{babel}
\usepackage{microtype}

\usepackage{amsmath,amssymb,amsfonts,amsthm,esint}
\usepackage{mathtools,accents}
\usepackage{mathrsfs}
\usepackage{aliascnt}
\usepackage{braket}
\usepackage{bm}

\usepackage[a4paper,margin=2cm]{geometry}
\usepackage[citecolor=blue,colorlinks]{hyperref}

\usepackage{enumerate}
\usepackage{xcolor}

\makeatletter
\g@addto@macro\@floatboxreset\centering
\makeatother

\makeatletter
\def\newaliasedtheorem#1[#2]#3{
	\newaliascnt{#1@alt}{#2}
	\newtheorem{#1}[#1@alt]{#3}
	\expandafter\newcommand\csname #1@altname\endcsname{#3}
}
\makeatother

\numberwithin{equation}{section}

\newtheoremstyle{slanted}{\topsep}{\topsep}{\slshape}{}{\bfseries}{.}{.5em}{}

\theoremstyle{plain}
\newtheorem{theorem}{Theorem}[section]
\newaliasedtheorem{proposition}[theorem]{Proposition}
\newaliasedtheorem{lemma}[theorem]{Lemma}
\newaliasedtheorem{corollary}[theorem]{Corollary}
\newaliasedtheorem{counterexample}[theorem]{Counterexample}

\theoremstyle{definition}
\newaliasedtheorem{definition}[theorem]{Definition}
\newaliasedtheorem{question}[theorem]{Question}
\newaliasedtheorem{openquestion}[theorem]{Open Question}
\newaliasedtheorem{conjecture}[theorem]{Conjecture}

\theoremstyle{remark}
\newaliasedtheorem{remark}[theorem]{Remark}
\newaliasedtheorem{example}[theorem]{Example}

\newcommand{\N}{\mathbb{N}}

\newcommand{\R}{\mathbb{R}}
\newcommand{\T}{\mathbb{T}}

\newcommand{\eps}{\varepsilon}

\let\altphi\phi
\let\phi\varphi
\let\varphi\altphi
\let\altphi\undefined

\newcommand{\abs}[1]{\left\lvert#1\right\rvert}
\newcommand{\norm}[1]{\left\lVert#1\right\rVert}

\let\div\undefined
\DeclareMathOperator{\div}{div}

\newcommand{\di}{\mathop{}\!\mathrm{d}}

\DeclareMathOperator{\supp}{supp}

\newfont{\tmpf}{cmsy10 scaled 2500}

\def\XXint#1#2#3{{\setbox0=\hbox{$#1{#2#3}{\int}$ }
		\vcenter{\hbox{$#2#3$ }}\kern-.6\wd0}}

\begin{document}
	
\title{Anomalous dissipation for the forced 3D Navier-Stokes equations}

\author{Elia Bru\`{e}, Camillo De Lellis}

\address{School of Mathematics, Institute for Advanced Study\\
	1 Einstein Dr.\\
	Princeton NJ 05840\\
	U.S.A.}
\email{elia.brue@math.ias.edu}

\address{School of Mathematics, Institute for Advanced Study\\
	1 Einstein Dr.\\
	Princeton NJ 05840\\
	U.S.A.}
\email{camillo.delellis@math.ias.edu}

\maketitle

\begin{abstract}
In this paper, we consider the forced incompressible Navier-Stokes equations with vanishing viscosity on the three-dimensional torus. We show that there are (classical) solutions for which the dissipation rate of the kinetic energy is bounded away from zero, uniformly in the viscosity parameter, while the body forces are uniformly bounded in some reasonable regularity class. 
\end{abstract}

\tableofcontents

\section{Introduction}

In this paper we study the vanishing viscosity limit for the incompressible $3d$ Navier-Stokes system namely we consider solutions $u^\nu: \T^3\times[0,1] \to \R^3$ of the system
\begin{equation}\label{NS}
	\begin{cases}
		\partial_t u^\nu + u^\nu \cdot \nabla u^\nu + \nabla p^\nu = \nu \Delta u^\nu + f^\nu\, ,
		\quad \text{on $\T^3\times[0,1]$}
		\\
		\div u^\nu = 0 \, ,
		\\
		u^\nu(\cdot, 0) = u_0^\nu \, .
	\end{cases}\tag{NS}
\end{equation}
In particular
$\nu>0$ is the viscosity, $u^\nu$ is the velocity of the fluid, $p^\nu: \T^3 \times [0,1] \to \R$ is the pressure, and $f^\nu:\T^3 \times [0,1] \to \R^3$ is an external body force. We deal with classical solutions, hence $u^\nu$, $p^\nu$ and $f^\nu$ will be assumed to be smooth throughout the paper and we also assume that the initial data $u_0^\nu$ satisfy a uniform $L^2$ bound (in other words the total kinetic energy at the initial time is bounded).

Since $u^\nu$ is divergence free, it is immediate to see that the kinetic energy decay is governed by
\begin{equation}\label{eq: energy balance}
	\frac{\di}{\di t} \frac{1}{2} \norm{u^\nu(\cdot, t)}_{L^2}^2 
	= - \nu \norm{\nabla u^\nu( \cdot, t)}_{L^2}^2 + \int_{\T^3} f^\nu(x,t)\cdot u^\nu(x,t) \di x \, .
\end{equation}
The first term appearing on the right-hand side is the total work of the force $f^\nu$, while the second term is the {\it energy dissipation rate} due to the viscosity of the fluid.

A fundamental postulate of Kolmogorov’s 1941 theory of fully developed turbulence \cite{Kolmogorov1,Kolmogorov2,Kolmogorov3}, called the {\it zeroth law of turbulence}, is that the anomalous dissipation of the kinetic energy holds, namely the inequality 
\begin{equation}\label{eq:anomalousdissintro} 
    \liminf_{\nu \to 0} \nu \int_0^T \| \nabla u^\nu(\cdot,s) \|_{L^2}^2 \di s > 0 \, 
\end{equation}
is valid for finite times $T$, even when the force $f^\nu$ excites only a finite number of Fourier modes and the sequence of initial data $u_0^\nu$ does not introduce itself microscopic scales. These two assumptions can be described loosely as a uniform regularity of the forces $f^\nu$ (uniform in the parameter $\nu$) and the absence of strong oscillations in the sequence $u_0^\nu$ (for instance precompactness of $\{u_0^\nu\}_\nu$ in the strong topology of $L^2 (\mathbb T^3)$).

Observe indeed that, if the sequence $f^\nu$ becomes very irregular as $\nu\downarrow 0$, or if $u^\nu_0$ converges weakly but not strongly in $L^2$, then even solutions of the linear Stokes equations (i.e. the system obtained by dropping the nonlinear term in the first equation of \eqref{NS})
would exhibit \eqref{eq:anomalousdissintro}. 
Instead, Kolmogorov's theory postulates that the creation of infinitely many scales and the cascade of energy through them is due to the quadratic nonlinearity $(u^\nu \cdot \nabla) u^\nu$.

The zeroth law of turbulence is verified experimentally to an enormous degree \cite{KIYIU,PKW,S98}, see also the recent review \cite{V15}, but to date, there are no known examples where it is rigorously proved in the framework described above. 
In experiments and simulations \cite{KIYIU,PKW,S98,V15} the force term is assumed to be injecting energy at {\it low frequencies}, meaning that $f^\nu$ is concentrated in frequency on a certain ball $\{k\in \mathbb{Z}^3\, : |k| \le \lambda\}$, uniformly in the viscosity parameter $\nu$.

Note however that, under the latter assumption, unless a certain amount of ``irregularity'' is introduced at the level of the initial data $u_0^\nu$, \eqref{eq:anomalousdissintro} can only hold if $T$ is larger than the first blow-up time of some suitable classical solution of the incompressible Euler equations. Indeed, assume that $f^\nu$ and $u^\nu_0$ enjoy uniform bounds in some space of smooth functions and hence converge, up to subsequences, to some sufficiently smooth $u_0, f$. By classical results, for some nontrivial finite interval $[0, T_0]$ there is a unique sufficiently smooth solution of the incompressible Euler (for instance this holds for any space which embeds in $C^{1,\alpha}$ for some $\alpha > 0$, cf. \cite{MB}). It is then relatively simple to show that $u^\nu$ converges strongly in $C ([0, T]; L^2(\T^3))$ to $u_0$ as long as on the interval $[0,T]$ the solution $u_0$ stays Lipschitz. From this strong convergence, it is then elementary to infer that \eqref{eq:anomalousdissintro} cannot hold (for the reader's convenience we include an elementary and short proof of these facts in the \autoref{Appendix}).  

The general principle is that the existence of a sufficiently regular solution of the incompressible Euler ensures the strong convergence to it of any ``reasonable'' regularization of Euler and in particular of Leray solutions of Navier-Stokes (see for instance the work \cite{BDS} for a precise formulation of this principle in rather general terms). In particular, proving \eqref{eq:anomalousdissintro} for sequences $u_0^\nu$ and $f^\nu$ which satisfy good uniform bounds would settle, as a corollary, the blow-up problem of incompressible Euler for force and initial data in a corresponding space of functions (on the negative). To date, the remarkable recent work of Elgindi \cite{Elgindi} is the only example of finite-time blow-up (in $\mathbb R^3$) in a space of functions for which there is local well-posedness of classical solutions of the incompressible Euler equations ($C^{1,\alpha}$ for some positive exponent $\alpha$). The problem of blow-up of classical solutions of Euler for more regular initial data and force (e.g. $C^2$) is still widely open. 

\medskip

Having established that \eqref{eq:anomalousdissintro} can only occur if a certain amount of regularity is lost in the vanishing viscosity limit, in this note we focus our attention on sequences of solutions for which the initial data $u_0^\nu$ enjoys the best possible bounds, i.e. it is fixed and smooth, while some (but as little as possible) irregularity is produced by the forcing terms $f$. Given that $C^1$ is a borderline regularity for local well-posedness of Euler, the above discussions suggest to investigate what happens if we assume the bounds
\begin{equation}\label{eq:reg_f 1}
   \sup_m  \| f^{\nu_m}\|_{C([0,1];C^\alpha(\T^3))} < \infty \, 
   \quad \text{for all $\alpha\in (0,1)$}\, ,
\end{equation}
or the bounds
\begin{equation}\label{eq:reg_f 2}
   \sup_m  \| f^{\nu_m}\|_{C([0,1];W^{1,p}(\T^3))} < \infty \, 
   \quad \text{for all $p<\infty$}\, ,
\end{equation}
for some sequence $\nu_m\downarrow 0$.

The first main theorem of this note states that, under the assumption \eqref{eq:reg_f 1} it is possible to rigorously show the occurrence of anomalous dissipation. The precise statement is given in the following theorem.

\begin{theorem}\label{thm:main1}
	There exist $\nu_m \downarrow 0$, $f^{\nu_m}\in C^\infty(\T^3 \times [0,1])$ and $u^{\nu_m}_0 \equiv u_0 \in C^\infty (\mathbb T^3)$ such that \eqref{eq:reg_f 1} holds and for which there is a unique smooth solution $u^{\nu_m} \in C^\infty(\T^3 \times [0,1])$ to \eqref{NS} with
	    \begin{equation}\label{eq:anomalous diss}
	    	\liminf_{m\to \infty} \nu_m \int_0^1\int_{\T^3} |\nabla u^{\nu_m}(x,t)|^2 \di x \di t > 0 \, .
	    \end{equation}
\end{theorem}

It is worth stressing that \eqref{eq:reg_f 1} (the weaker of the two) is strong enough to rule out the occurrence of anomalous dissipation when we drop the nonlinear term from \eqref{NS}. More precisely, if we consider the linear Stokes equations with a fixed smooth initial data $u_0$ and a family of forces $f^\nu$, we can use the Stokes semigroup $P_t^{\nu}$ to represent the solution as 
\begin{equation}
    u^\nu(x,t) = P_t^{\nu} u_0(x) + \int_0^t P_{t-s}^{\nu} f^\nu(\cdot,s)(x)\, \di s \, 
\end{equation}
and hence derive the uniform bound
\begin{equation}
    \nu \int_0^1 \| \nabla u^\nu(\cdot, s)\|_{L^2}^2\, \di s
    \le 
    \nu \| \nabla u_0 \|_{L^\infty(\T^3\times [0,1])}^2 + \nu^{\alpha} C(\alpha) \| f^\nu \|_{C([0,1];C^\alpha(\T^3))}^2 = O(\nu^\alpha)\, \, .
\end{equation}

If we approach the threshold $C^1$ on the Sobolev side, we cannot show \eqref{eq:anomalousdissintro} but we can prove a rather strong form of \textit{enhanced dissipation}.

\begin{theorem}\label{thm:main2}
    	There exist $\nu_m \downarrow 0$, $f^{\nu_m}\in C^\infty(\T^3 \times [0,1])$, and $u^{\nu_m}_0 \equiv u_0\in C^\infty (\mathbb T^3)$ such that \eqref{eq:reg_f 2} holds and for which there is a unique smooth solution $u^{\nu_m} \in C^\infty(\T^3 \times [0,1])$ with the property that 
	    \begin{equation}\label{eq:enhanced with rate}
	   \nu_m \int_0^1\int_{\T^3} |\nabla u^{\nu_m}(x,t)|^2 \di x \di t 
	   \ge 
	   C \exp\{ -\log^{2/3}(1/\nu_m) \}\, ,
\end{equation}
for some constant $C$ independent of $m$.
\end{theorem}
Note in particular that, in the latter case, the energy dissipation rate per unit mass decreases slower than any power law, namely for any $\theta>0$ there is $C= C(\theta)$ such that
\begin{equation}\label{eq:enhanced diss power}
	    \nu_m \int_0^1\int_{\T^3} |\nabla u^{\nu_m}(x,t)|^2 \di x \di t
	    \ge C \nu^\theta_m \, .
	    \end{equation}

In the next section, we discuss related open problems and possible lines of research, while in \autoref{sec:strategy} we outline the main ideas of the arguments which take advantage of results and techniques from the previous works \cite{AlbertiCrippaMazzucato16,DrivasElgindiIyerJeong19,JeongYoneda21}.

\subsection*{Acknowledgments}
EB is supported by  the Giorgio and Elena Petronio Fellowship at the Institute for Advanced Study. The research of CDL has been funded by the NSF under grant DMS-1946175. 
The authors are grateful to Liceo Scientifico G. Leopardi and Hotel La Ginestra in Recanati for their hospitality in early July when this paper was completed.

\section{Further comments}

Theorems \ref{thm:main1} and \ref{thm:main2} raise a series of interesting questions.

\begin{question}\label{q:viscosity-independent}
Is it possible to make the sequence $f^{\nu_m}$ independent of the viscosity parameter $\nu_m$?
\end{question}

In the latter case we mean that $f^{\nu_m}$ would be equal to a single $f\in \bigcap_{\alpha <1} C^\alpha$: for finite positive viscosity $\nu_m >0$ we could then expect to have $C^\alpha$ solutions, a degree of regularity which is enough to consider the solutions ``classical'', given the regularity theory for the incompressible Navier-Stokes in three dimensions.

\begin{question}\label{q:time-independent}
Is it possible to make the sequence of forces $f^{\nu_m}$ independent of time?
\end{question}

The above questions are interesting even in higher space dimensions. In fact we suspect that a positive answer to Question \ref{q:viscosity-independent} on $\mathbb T^3$ would give a positive answer to Question \ref{q:time-independent} on $\mathbb T^4$. 

As it will be clear from the arguments in the next sections, the solutions $u^{\nu_m}$ provided by the proofs in \autoref{thm:main1} and \autoref{thm:main2} converge both to very regular solutions of incompressible Euler with a {\em smooth} limiting force $f$ on the open interval $[0,1)$. More precisely in the case of \autoref{thm:main1} there exists
$u\in C^\infty(\T^3\times [0,1))\cap L^\infty(\T^3\times [0,1])$, solution to the forced Euler equations  
\begin{equation}\label{e:forced-Euler}
\left\{\begin{array}{l}
    \partial_t u + u\cdot \nabla u + \nabla p = f \, ,\\
    {\rm div}\, u = 0
\end{array}
\right.
\end{equation}
such that $u^{\nu_m} \to u$ weakly in $L^2(\T^2\times [0,1])$, as $m \to \infty$.
Up to subsequences we can always assume the strong convergence
$f^{\nu_m} \to f$ in $C([0,1];C^\alpha(\T^3))$ for all $\alpha\in (0,1)$. The anomalous dissipation proved in \autoref{thm:main1}, and the energy balance \eqref{eq: energy balance}, imply that $u$ violates the energy equality at time $t=1$:
\begin{equation}
    \| u(\cdot,1) \|_{L^2}^2 < \| u(\cdot,0) \|_{L^2}^2 + \int_0^1\int_{\T^3} f(x,s) \cdot u(x,s)\, \di x\di s \, .
\end{equation}
However, the energy equality is satisfied for all smaller times. In particular there is a ``sudden drop'' in the kinetic energy at time $t=1$. For this reason, it is natural to ask the following question.

\begin{question}\label{q:slow-dissipation}
Is it possible to produce a sequence $u^{\nu_m}$ as in Theorem \ref{thm:main1} which converges to a weak solution $u$ of the forced Euler equation \eqref{e:forced-Euler} for which $t\mapsto \|u(\cdot, t)\|_{L^2}$ is continuous on $[0,1]$ and 
\begin{equation}\label{e:Euler-dissip}
    \| u(\cdot,1) \|_{L^2}^2 < \| u(\cdot,0) \|_{L^2}^2 + \int_0^1\int_{\T^3} f(x,s) \cdot u(x,s)\, \di x\di s \, .
\end{equation}
\end{question}

We remark in passing that for the example given in our proof  of \autoref{thm:main2}, the convergence $u^{\nu_m} \to u$ is strong in $L^p(\T^3)$ for any $p<\infty$, while $u\in C^\infty(\T^3\times [0,1))\cap L^\infty(\T^3\times [0,1])$, and moreover it is possible to show that the limiting $u$ does satisfy the energy balance
\begin{equation}
    \| u(\cdot,t) \|_{L^2}^2 = \| u(\cdot,0) \|_{L^2}^2 + \int_0^t\int_{\T^3} f(x,s) \cdot u(x,s)\, \di x \di s\, 
\end{equation}
for any time $t$, including $t=1$. 

\medskip

Coming back to the anomalous dissipation, it is tempting to introduce an analog of the famous Onsager conjecture, cf. \cite{Onsager}, proved by Isett in \cite{Isett}, for the forced Euler equations. First of all it is not difficult to see that the proof given in \cite{CET} by Constantin, E, and Titi of the positive part of the Onsager conjecture implies in fact that the conservation of energy 
\begin{equation}
    \| u(\cdot,t) \|_{L^2}^2 = \| u(\cdot,0) \|_{L^2}^2 + \int_0^1\int_{\T^3} f(x,s) \cdot u(x,s)\, \di x\di s \, 
\end{equation}
holds for weak solutions of \eqref{e:forced-Euler} under the assumptions that 
\begin{itemize}
    \item[(a)] $u\in L^p([0,1];C^\alpha)$ for some $\alpha >\frac{1}{3}$ and some $p\geq 3$;
    
    \item[(b)] $f\in L^{p'}([0,1];C^{-\alpha})$\footnote{Here we denote by $C^{-\alpha} (\mathbb T^3)$ the linear space of distributions $T$ which is the dual of $C^\alpha (\mathbb T^3)$, i.e. those distributions $T$ which satisfy the linear inequality $|T (\varphi)|\leq C \|\varphi\|_\alpha$ for any test $\varphi\in C^\infty (\mathbb T^3)$. Given the latter estimate the action of such distribution can be extended in a unique way to any test $\varphi \in C^\alpha (\mathbb T^3)$. By a slight abuse of notation we keep writing $\int T \varphi$ instead of $T (\varphi)$.} for the dual exponent $p' = \frac{p-1}{p}$. 
\end{itemize}
On the other hand, such a weak assumption as (b) is not compatible with the energy class $u\in L^\infty([0,1];L^2(\T^3))$. A natural replacement would be $f\in L^1([0,1];L^2(\T^3))$.

The Onsager conjecture has been proved using ``convex integration'' techniques (introduced in the context in \cite{DS1,DS2}) and it is clear that most of the dissipative solutions produced by these methods cannot arise as vanishing viscosity limits of {\em Leray} solutions of Navier-Stokes (the remarkable paper \cite{BV} shows however that in a vast majority of cases they are vanishing viscosity limits of {\em weak} (or Oseen) solutions of Navier-Stokes). Even showing that some solution produced by convex integration is the limit of a classical vanishing viscosity approximation is a widely open problem. Producing examples for the {\em forced} Euler and Navier-Stokes which validate the Onsager threshold might be a more tractable problem.

\begin{question}\label{q:Onsager}
Let $\alpha$ be any positive number smaller than $\frac{1}{3}$. Is it possible to produce a sequence $\nu_m\downarrow 0$ and two sequences $u^{\nu_m}$ and $f^{\nu_m}$ of smooth solutions of \eqref{NS} with the following properties:
\begin{itemize}
    \item[(1)] $\sup_m \|u^{\nu_m}\|_{L^3 ([0,1]; C^\alpha)} < \infty$;
    \item[(2)] $\sup_m \|f^{\nu_m}\|_{L^{1+\eps} ([0,1]; C^{\eps})} < \infty$, for some $\eps>0$;
    \item[(3)] \eqref{eq:anomalousdissintro} holds.
\end{itemize}
\end{question}
Given the uniform estimates (1) and (2) it is a simple exercise to show that, up to subsequences, $u^{\nu_m}$ and $f^{\nu_m}$ would then converge to a pair $u$ and $f$ solving \eqref{e:forced-Euler} in the sense of distributions. If $\alpha>\frac{1}{3}$, an argument analogous to the one presented in \cite{DrivasEyink} shows that $u^{\nu_m}$ does not display anomalous dissipation.

Under assumption (2) complemented with $f^{\nu_m} \to f$ in $L^{1+\eps}([0,1]; C^\eps(\T^3))$ and $\sup_m \| u_0^{\nu_m}\|_{C^\eps}<\infty$, it is possible to see that the corresponding solutions to the linear Stokes equations do not exhibit anomalous dissipation. Indeed, Duhamel's identity gives a uniform $C([0,1];C^\eps(\T^3))$ bound on the solutions $u^{\nu_m}$, an Aubin-Lions' type lemma gives the strong convergence of the latter in $C([0,1];L^2(\T^3))$ to a solution $u\in C([0,1];L^2(\T^3))$ satisfying the energy equality
\begin{equation}
    \| u(\cdot, t) \|_{L^2}^2 = \|u_0 \|_{L^2}^2 + \int_0^t \int f(x,s)\cdot u(x,s) \, dx\, ds \, .
\end{equation}


\section{Strategy of the proof}\label{sec:strategy}

Our construction is achieved in the framework of {\it$(2+\frac{1}{2})$-dimensional flows},
where the evolution reduces to a $2d$-NS system coupled with a scalar advection-diffusion equation.
This framework has already been considered in the study of anomalous dissipation by Jeong and Yoneda in \cite{JeongYoneda21,JeongYoneda22}.

\subsection{The \texorpdfstring{$(2+\frac{1}{2})$}{2+1/2}-dimensional flow}
We consider $(u^\nu, p^\nu, f^\nu)$, solutions to \eqref{NS} admitting the following structure: 
\begin{equation}\label{eq:2+1/2}
	\begin{split}
		u^\nu(x,t) & = ( v^\nu(x_1,x_2,t), \theta^\nu(x_1,x_2,t))
		\\
		p^\nu(x,t) & = (q^\nu(x_1,x_2,t),0)
		\\
		f^\nu(x,t) & = (g^\nu(x_1,x_2,t), 0)
	\end{split}	
\end{equation}
where $(x_1,x_2,x_3)=x\in \T^3$,  $v^\nu, g^\nu : \T^2 \times [0,1] \to \R^2$ are vector fields, and $\theta^\nu, q^\nu : \T^2 \times [0,1] \to \R$ are scalars.
It turns out that the structure \eqref{eq:2+1/2} is conserved along the \eqref{NS} evolution, hence we have the system:
\begin{equation}\label{eq:NS 2+1/2}
	\begin{cases}
		\partial_t v^\nu +  v^\nu \cdot \nabla v^\nu + \nabla q^\nu = \nu \Delta v^\nu + g^\nu
		\\
		\div v^\nu = 0
		\\
		\partial_t \theta^\nu + v^\nu \cdot \nabla \theta^\nu = \nu \Delta \theta^\nu \, \, .
	\end{cases}\tag{$(2+\frac{1}{2})$-NS}
\end{equation}
As for the initial conditions, they will be $\nu$-independent and therefore we will set them to be 
\begin{align}
v^\nu(x_1, x_2,0) & = v_0(x_1, x_2)\\
\theta^\nu(x_1, x_2 ,0) &= \theta_0(x_1, x_2)\, ,
\end{align}
where $\theta_0$ and $v_0$ are, respectively, a smooth scalar function and a smooth $2$-dimensional vector field on $\mathbb T^2$.

We will now state more precise versions of \autoref{thm:main1} and \autoref{thm:main2}. 

\begin{theorem}\label{thm:main1 version 2}
	There exist a sequence $\nu_m \downarrow  0$, $g^{\nu_m} \in C^\infty(\T^2 \times [0,1])$, and $v_0,\theta_0\in C^\infty(\T^2)$ satisfying the following properties:
	\begin{itemize}
		\item[(1)]
		There exists $g\in \bigcap_{\alpha\in (0,1)} C([0,1]; C^\alpha(\T^2))$, such that $g^{\nu_m}\to g$ in $C([0,1]; C^\alpha(\T^2))$ for any $\alpha\in (0,1)$.

		\item[(2)] The smooth solution $v^{\nu_m}, \theta^m \in C^\infty(\T^2 \times [0,1])$ to \eqref{eq:NS 2+1/2} displays anomalous dissipation, i.e.
		\begin{equation}
			\liminf_{m\to \infty} \nu_m \int_0^1\int_{\T^2} |\nabla \theta^{\nu_m}(x,t)|^2 \di x \di t > 0 \, .
		\end{equation}
	\end{itemize}
\end{theorem}
We are able to show that the velocity fields $v^{\nu_m}$ enjoys the same regularity of the body force $g^{\nu_m}$. More precisely, $v^{\nu_m}\in C([0,1],C^\alpha(\T^3))$ for any $\alpha\in (0,1)$, uniformly in $m$. In particular, in the vanishing viscosity limit $\nu_m \downarrow 0$, the $(2+\frac{1}{2})$-structure is preserved and $u(x,t) = (v(x_1,x_2,t) ,\theta(x_1,x_2,t))$, where $v\in C([0,1],C^\alpha(\T^3))$, and $\theta \in L^\infty(\T^2\times [0,1])$.

\begin{theorem}\label{thm:main2 version 2}
		There exist a sequence $\nu_m \downarrow  0$, $g^{\nu_m} \in C^\infty(\T^2 \times [0,1])$, and $v_0,\theta_0\in C^\infty(\T^2)$ satisfying the following properties:
	\begin{itemize}
		\item[(1)]
		There exists $g\in \bigcap_{p<\infty} C([0,1]; W^{1,p}(\T^2))$, such that $g^{\nu_m}\to g$ in $C([0,1]; W^{1,p}(\T^2))$ for any $p<\infty$.

		\item[(2)] The smooth solution $v^{\nu_m}, \theta^m \in C^\infty(\T^2 \times [0,1])$ to \eqref{eq:NS 2+1/2} satisfies
		\begin{equation}\label{eq:enhanced with rate2}
    \nu_m \int_0^1\int_{\T^2} |\nabla \theta^{\nu_m}(x,t)|^2 \di x \di t
    \ge C \exp\{ -\log^{3/2}(1/\nu_m) \}
	   \quad \text{for every $m\in \mathbb{N}$}\, ,
\end{equation}
for some constant $C$ independent of $m$.
	\end{itemize}
\end{theorem}

\begin{remark} In fact, for the specific sequence constructed in the proof of the above theorem we are able to prove that $v^{\nu_m}$ enjoys uniform in $m$ bounds in the space $C([0,1], W^{1,p}(\T^3))$ for any $p<\infty$, uniformly in $m$. In particular, $v^{\nu_m}$ falls in the framework of the DiPerna-Lions theory \cite{DiPernaLions,Ambrosio04}, excluding the possibility of anomalous dissipation. More precisely, following \cite[Theorem 0.4]{BrueNguyen19} we get that the energy dissipation rate per unit mass must decrease at a logarithmic rate, namely for every $p\geq 1$ there is a constant $C(p)$ such that 
 \begin{equation}
     \nu_m \int_0^1\int_{\T^2} |\nabla \theta^{\nu_m}(x,t)|^2 \di x \di t
     \le C(p) \log^p(1/\nu_m) \, .
 \end{equation}
The latter estimate shows that \eqref{eq:enhanced with rate2} cannot be meaningfully improved following our strategy.
\end{remark}

\subsection{Quasi-self-similar evolution}
The key idea of proof of \autoref{thm:main1 version 2} and \autoref{thm:main2 version 2} is to employ a {\it quasi-self-similar} evolution to produce, simultaneously, small scales in $\theta^{\nu_m}$ and a controlled body force $g^{\nu_m}$. 

Let us consider $(V(x,t), \Theta(x,t)$) solving the transport equation
\begin{equation}
    \partial_t \Theta + V \cdot \nabla \Theta = 0\, , \quad \text{in $\T^2\times [0,1]$} \, ,
\end{equation}
and assume that $\Theta$ satisfies
\begin{equation}\label{eq:self-similar intro}
    \Theta(x,1) = \Theta(5x,0)\, , \quad x\in \T^2\, .
\end{equation}
 Assume now for simplicity that we can find such a nontrivial pair $V$ and $\Theta$ which are both smooth. This is in fact a very strong assumption, to date we are not aware of any nontrivial pair $V$ and $\Theta$ with the latter property. If it were possible to show their existence we could give a purely self-similar evolution which enjoys stronger regularity properties (where the self-similarity must be understood as a ``discrete self-similarity'', cf. \eqref{e:self-similar} - \eqref{e:self-similar2}). 
 
 In the actual proof we thus resort
 to a quasi-self-similar evolution in the actual construction, while for the sake of this discussion we stick to the purely self-similar setting. Starting from $(V,\Theta)$ we could build a rapid self-similar evolution as
\begin{align}
    v(x,t) & = \sum_{n\ge 0} \chi_{[t_n, t_{n+1})}(t) \frac{1}{5^n} \frac{1}{t_n- t_{n+1}} V\left(5^n x, \frac{t-t_n}{t_{n+1}-t_n}\right)\label{e:self-similar}
    \\
    \theta(x,t) & = \sum_{n\ge 0} \chi_{[t_n, t_{n+1})}(t) \Theta \left(5^n x, \frac{t-t_n}{t_{n+1}-t_n}\right)\label{e:self-similar2}
\end{align}
where $t_n = 1-(n+1)^{-2}$ and $\chi_A$ denotes the indicator function of the set $A$. The key observation is that $(v,\theta)$ solves the transport equation and
\begin{equation}
    \theta(x,t) \sim \Theta(5^n x, 0) \, , \quad \text{when $t\in (t_n,t_{n+1})$}\, ,
\end{equation}
hence, the transport evolution creates small scales very quickly: any Sobolev norm of $\theta$ blows up at time $t=1$. On the other hand, the velocity field $v(x,t)$, when plugged into the non-linear term of the Euler equations, produces a term of roughly the same size:
\begin{equation}
    (v\cdot \nabla v) (x,t) = \sum_{n\ge 0} \chi_{[t_n,t_{n+1})}(t) \frac{1}{5^n}\frac{1}{(t_n- t_{n+1})^2}(V\cdot \nabla V)  \left(5^n x, \frac{t-t_n}{t_{n+1}-t_n}\right) \, .
\end{equation}
This consideration suggests therefore the following construction.
We first regularize $v$ and $\theta$ by stopping the evolution to time $t_m<1$, and mollifying the time variable. The resulting solution to the transport equation $(v^m, \theta^m)$ is smooth and close to $(v,\theta)$ when $m$ is big enough. We then define $\nu_m\ll 1$ such that $v^{\nu_m} := v^m$, when plugged into the Euler equations, produces a force term $g^m$ with roughly the same size of $v^m$. We then solve the advection diffusion equation
\begin{equation}
   \begin{cases}
    \partial_t \theta^{\nu_m} + v^{\nu_m}\cdot \nabla \theta^{\nu_m} = \nu_m \Delta \theta^{\nu_n}\, ,
    \\
    \theta^m(x,0) = \theta(x,0)\in C^\infty(\T^2)\, ,
    \end{cases}
\end{equation}
and use that $\theta^{\nu_m} \sim \theta^m \sim \theta$ in certain regimes, hence the small scales of $\theta$ can be used to produce anomalous dissipation.

\section{Quasi-self-similar evolution for passive scalars}\label{sec:quasi self similar}

In this section we describe the construction of a family of smooth {\it quasi-self-similar} solutions to the transport equation, first obtained in \cite{AlbertiCrippaMazzucato16}. They are a smooth replacement of purely self-similar evolutions, which share several structural properties with the latter. In the sequel we will employ the quasi self-similar family to build two different solutions to the forced $2d$ Navier-Stokes equation. The first one produces anomalous dissipation as in \autoref{thm:main1 version 2}, the second one which is more regular, is used in the proof of \autoref{thm:main2 version 2}.

\medskip

Let us begin by introducing some notation. For any $\lambda\in \mathbb{N}\setminus \{0,1\}$ we denote by $\mathcal{Q}(\lambda)$ the family of open squares in $[0,1]^2$ with sidelength $\lambda^{-1}$ and vertices in $\lambda^{-1} \mathbb Z^2 \cap [0,1]^2$. For any $Q\in \mathcal{Q}(\lambda)$ we let $r(Q)\in \mathbb{Q}$ be such that $Q - r(Q) = (0, \lambda^{-1})^2$.

\subsection{Building blocks}\label{subsec:building block}

Let us fix an integer $N>1$. We consider a family of velocity fields $V_1(x,t), \ldots, V_N(x,t)$ and scalars $\Theta_1(x,t), \ldots, \Theta_N(x,t)$ satisfying the following properties for any $i=1, \ldots, N$:
\begin{itemize}
	\item[(i)] $V_i\in C^\infty([0,1]^2 \times [0,1]; \R^2)$ is divergence free and tangent to the boundary $\partial [0,1]^2$;

	\item[(ii)] $\Theta_i\in C^\infty([0,1]^2 \times [0,1])$ is non-constant and satisfies the moment conditions \begin{align*}
    \int_{(0,1)^2}\Theta_i(x,t)\di x  &= 0\\ 
    \int_{(0,1)^2} \Theta_i (x,t)^2\di x &=1
    \end{align*}
    for every $t\in [0,1]$;

	\item[(iii)] $(V_i, \Theta_i)$ is a solution to the transport equation, i.e.
	\begin{equation}
		\partial_t \Theta_i + V_i \cdot \nabla \Theta_i = 0 \, \qquad \text{in $[0,1]^2\times [0,1]$}\, ;
	\end{equation}
    		
    \item[(iv)] for every $Q\in \mathcal{Q}(5)$ there exists $j=j(Q,i)\in \{1, \ldots, N\}$ such that
    \begin{equation}
    	\Theta_i(x,1) = \Theta_j(5(x-r(Q)), 0) \, ,
    	\quad \text{for every $x\in Q$}\, .
    \end{equation}	    
\end{itemize} 
In other words, $\Theta_i(\cdot,1)$ can be realized by patching together elements of the family $\{\Theta_j(\cdot, 0)\}_{1\leq j \leq N}$ after rescaling them in space by a factor $5$. This is a clear generalzation of the notion of self-similar evolution where $\Theta(x,1)=\Theta(5x,1)$ for $x\in \mathbb{T}^2$.

\subsection{Quasi-self-similar family}
Starting from a family of building blocks satisfying (i)-(iv), and an extra compatibility condition, we build a smooth family of quasi-self-similar solutions to the transport equation.
The following result is taken from \cite[Section 8]{AlbertiCrippaMazzucato16}.

\begin{theorem}\label{thm: quasi-self-sim}
	There exist $V_i$ and $\Theta_i$ satisfying (i)-(iv) with $N=6$.
	They can be patched together to form a quasi-self-similar evolution, i.e. a family $\{(\rho_n(x,t), v_n(x,t))\, : \, n \in \mathbb{N}\}$ of smooth solutions to the transport equation in $[0,1]^2 \times [0,1]$ with the following structure
	\begin{align}\label{eq:rho_n v_n}
		\rho_n(x,t) & = \sum_{Q\in \mathcal{Q}(2\cdot 5^n)} \chi_{Q}(x) \Theta_{i(Q)}(2\cdot 5^n (x - r(Q)) ,t)
		\\
		v_n(x,t) & = \sum_{Q\in \mathcal{Q}(2\cdot 5^n)} \chi_Q(x) \frac{1}{2\cdot 5^n}V_{i(Q)}(2\cdot 5^n (x - r(Q)) ,t)  \, .
	\end{align}
	Moreover, they satisfy the following properties for every $n\in \mathbb{N}$: 
	\begin{itemize}
		\item[(a)] $v_n\in C^\infty([0,1]^2\times [0,1];\R^2)$ is divergence free,  and
		\begin{equation}\label{eq:est1}
			\|  \partial_t^k v_n \|_{L^\infty([0,1];C^\alpha(\T^2))} \le C(\alpha,k)\, 5^{(\alpha-1) n} \, \quad \text{for every $\alpha\ge0$ and $k\in \mathbb{N}$}\, ;
		\end{equation}
		\item[(b)] $\rho_n \in C^\infty([0,1]^2\times [0,1])$,
		\begin{align*}
		\int_{(0,1)^2} \rho_n(x,t)\di x & = 0\\
		\int_{(0,1)^2} |\rho_n(x,t)|^2\di x &=1,
		\end{align*}
		(for every $t\in [0,1]$ and every $n$) and there is a constant $C$ such that
		\begin{align}
		\|\rho_n(\cdot,t)\|_{L^\infty} &\le 10\label{e:b2-1}\\  
		\|\nabla \rho_n(\cdot,t)\|_{L^\infty} &\le C\, 5^n\label{e:b2-2}\\ 
		\|\rho_n(\cdot,t)\|_{\dot H^{-1}} &\le C\, 5^{-n}\label{e:b2-3}
		\end{align}
		for every $t\in [0,1]$ and for every $n$;

		\item[(c)] there exists a compact set $K\subset (0,1)^2$ such that 
		\[
		\supp v_n(\cdot,t) \cup \supp \rho_n(\cdot, t)\subset K
		\]
		for any $n\in \mathbb{N}$ and every $t\in [0,1]$;
		
		\item[(d)] $\rho_n(x,1) = \rho_{n+1}(x,0)$ for every $n\in \mathbb N$
		
	\end{itemize}	
\end{theorem}

They main difficulty in the proof of \autoref{thm: quasi-self-sim} is to ensure that the quasi-self-similar family $(\rho_n, v_n)$ is regular in space.
Given any family of building blocks as in \autoref{subsec:building block}, one can always build a quasi-self-similar evolution through \eqref{eq:rho_n v_n}, by suitably choosing the indexes $i(Q)$. The problem is the presence of discontinuities along the boundaries of $Q\in \mathcal{Q}(2\cdot 5^n)$. To get around this, one has to carefully choose the family of building blocks and place them ensuring that two adjacent blocks coincide in a neighborhood of the interface. This requires a delicate combinatorial construction which is explained in detail in \cite[Section  8]{AlbertiCrippaMazzucato16}.

\subsection{Quasi-self-similar velocity field and Euler equation}

The following simple observation will be used several times throughout the paper. Let $v_n$ a quasi-self-similar velocity field as in \autoref{thm: quasi-self-sim}, if we plug it in the non-linear term of the Euler equations we get a body force whose size is comparable to $v_n$ in any H\"older space $C^\alpha$. More precisely, relying on \eqref{eq:rho_n v_n} and the fact that $V_{i(Q)}(2\cdot 5^n (x - r(Q)) ,t)$ is tangent to $\partial Q$ we deduce
    \begin{equation}\label{z1}
    	(v_n \cdot \nabla v_n)(x,t) 
    	= \sum_{Q\in \mathcal{Q}(2\cdot 5^n)} \chi_Q(x) \frac{1}{2\cdot 5^n}V_{i(Q)}\cdot \nabla V_{i(Q)}(2\cdot 5^n (x - r(Q)) ,t)\, .
    \end{equation}
     Hence, using that $V_{i(Q)}$ and $V_{i(Q')}$ coincide in a neighborhood of $\partial Q \cap \partial Q'$ when the latter is not empty (cf. the discussion after \autoref{thm: quasi-self-sim}), we deduce
    \begin{equation}\label{eq:non-lin est}
    \|\partial_t^k (v_n \cdot \nabla v_n)\|_{C([0,1];C^\alpha(\T^2))}
    \le C(\alpha,k) 5^{-(1-\alpha)n}\, ,
    \end{equation}
    for any $\alpha>0$ and $k\in \mathbb{N}$.

\section{Quasi-self-similar solutions to the forced \texorpdfstring{$2d$}{2d}-NS: first construction}
\label{sec: first construction}

We employ the family of quasi-self-similar evolutions built in \autoref{sec:quasi self similar} to produce a solution to the $2d$ Navier-Stokes equations with a sufficiently regular body force.
We will show in \autoref{sec:anomalous diss} that the associated advection-diffusion equation displays anomalous dissipation, hence concluding the proof of \autoref{thm:main1 version 2}.

\medskip

Let us consider $\{(\rho_n,v_n)\, : \, n\in \N\}$ as in \autoref{thm: quasi-self-sim}.
Thanks to (a) and (b) we can make both $\rho_n$ and $v_n$ $1$-periodic, hence defined on the $2d$-torus $\T^2$.

Given the sequence of times $t_n := 1 - (n+1)^{-2}$ we define
\begin{align}
	\tilde \rho(x,t) := & \sum_{n\ge 0} \chi_{[t_n, t_{n+1})}(t) 
	\rho_n \left(x, \frac{t-t_n}{t_n - t_{n+1}}\right)
	\\
	\tilde v(x,t) := & \sum_{n\ge 0} \chi_{[t_n, t_{n+1})}(t) \frac{1}{t_{n+1}-t_n} v_n
	\left(x, \frac{t-t_n}{t_n - t_{n+1}}\right) \, .
\end{align} 
It turns out that $(\tilde \rho, \tilde v)$ solves the transport equation in $\T^2\times [0,1]$, and both $\tilde \rho$ and $\tilde v$ are smooth in space for any $t\in [0,1)$.
However, they are only piecewise smooth in time. To solve this issue we apply the following standard trick. We consider a smooth non-decreasing function $\eta: [0,1]\to [0,1]$ satisfying the following properties:
\begin{itemize}
    \item[(1)] $\eta(t_n) = t_n$ for any $n\in \mathbb{N}$;
    \item[(2)] $\frac{\di^k}{\di t^k} \eta(t_n) = 0$ for any $n, k\in \mathbb{N}$, $k\ge 1$;
    \item[(3)] $|\frac{\di^k}{\di t^k} \eta(t)|\chi_{[t_n,t_{n+1})} \le C(k) n^{5k}$  for any $n, k\in \mathbb{N}$, and $t\in [0,1]$.
\end{itemize}
We then define
\begin{align}\label{eq: v, rho}
	\rho(x,t)  & = \tilde \rho(\eta(t),x)
	\\
	v(x,t) & = \eta'(t) \tilde v(\eta(t),x) \, .
\end{align}
It is immediate to see that $(\rho, v)$ solves the transport equation, and $\partial_t^k v\in L^\infty(\T^2 \times [0,1])$ for any $k\in \mathbb{N}$.

\subsection{\texorpdfstring{$2d$}{2d}-NS with body force}\label{subsec:2d-NS}
Let us begin by smoothing out the vector field $v(x,t)$ around $t=1$. For any integer $m\ge 2$ we define
\begin{equation}\label{eq: vm}
	v^m(x,t):= \sum_{n = 0}^m \eta'(t) \chi_{[t_n, t_{n+1})}(\eta(t)) \frac{1}{t_{n+1}-t_n} v_n
	\left(x, \frac{\eta(t)-t_n}{t_{n+1} - t_{n}}\right) \, ,
\end{equation}
and the viscosity parameter
\begin{equation}\label{eq:nu m}
    \nu_m := m^{10}5^{-2m} \, .
\end{equation}
We then introduce the body force generated by $v^m$ in the Navier-Stokes equations:
\begin{equation}\label{eq:2dNS}
	g^m := \partial_t v^m + v^m \cdot \nabla v^m - \nu_m \Delta v^m \, .
\end{equation}

\begin{lemma}\label{lemma: bound on g}
   Let $\nu_m$ and $g^m$ be as above.
   For any $\alpha\in (0,1)$ there is a constant $C(\alpha)$ such that
   \begin{equation}\label{eq: force est}
   	\|g^m\|_{C([0,1];C^\alpha(\T^2))} 
   	\le C(\alpha)\, .
   \end{equation}
   Moreover, $g^m \to g$ in $C([0,1];C^\alpha(\T^2))$ as $m \to \infty$, where 
   \begin{equation}
   	g = \partial_t v + v\cdot \nabla v
   \end{equation}
   is the body force generated by $v$ in the Euler equations.
\end{lemma}

\begin{proof}
   Let us begin by estimating $\| \partial_t v^m\|_{C([0,1];C^\alpha(\T^2))}$. Thanks to \eqref{eq:est1}, we have
\begin{align}
   \|\partial_t v(\cdot,t)&\|_{C^\alpha(\T^2)}  
   \le \sum_{n\ge 0} |\partial_t (\eta'(t) \chi_{[t_n, t_{n+1})}(\eta(t)) )|
   \frac{1}{t_{n+1}-t_n}
   \norm{ v_n
   \left(\cdot , \frac{\eta(t)-t_n}{t_n - t_{n+1}}\right)}_{C^\alpha(\T^2)}
   \\& \qquad \qquad
   + \sum_{n\ge 0} |\eta'(t)|^2\chi_{[t_n, t_{n+1})}(\eta(t))\frac{1}{(t_{n+1}-t_n)^2}
   \norm{\partial_t v_n \left(\cdot , \frac{\eta(t) - t_n}{t_n - t_{n+1}}\right)}_{C^\alpha(\T^2)}
   \\& \le
   C(\alpha)\sum_{n\ge 0} \left( \abs{\partial_t (\eta'(t) \chi_{[t_n, t_{n+1})}(\eta(t)) )}
   + |\eta'(t)|^2 \chi_{[t_n, t_{n+1})}(\eta(t)) \right) n^6 5^{-n(1-\alpha)}
    \, .
\end{align}
Notice that
\begin{equation}
	\abs{\partial_t (\eta'(t) \chi_{[t_n, t_{n+1})}(\eta(t)) )}
	= | \eta''(t)| \chi_{[t_n, t_{n+1})}(\eta(t))
	\, ,
\end{equation}
since $\eta'(t_n) = 0$ and $\eta(t_n) = t_n$. Moreover, from the property (3) of $\eta$ we deduce
\begin{equation}
	\chi_{[t_n, t_{n+1})}(\eta(t))
	( |\eta'(t)|^2+|\eta''(t)|) \le C n^{10}\, ,
\end{equation}
hence,
	\begin{equation}\label{z4}
		 \|\partial_t v^m \|_{C([0,1];C^\alpha(\T^2))}
		 \le 
		\|\partial_t v \|_{C([0,1];C^\alpha(\T^2))}
		\le C(\alpha) \, .
	\end{equation}
Moreover, it is clear from the previous estimate that
    \begin{equation}\label{z2}
    	\lim_{m\to \infty}  
    	\|\partial_t v^m - \partial_t v\|_{C([0,1];C^\alpha(\T^2))}
    	 = 0 \, .
    \end{equation}
    Let us now study the nonlinear term:
    \begin{align}
    	v^m \cdot \nabla v^m (x,t)
    	= \sum_{n=0}^m |\eta'(t)|^2 \chi_{[t_n, t_{n+1})}(\eta(t)) \frac{1}{(t_{n+1}-t_n)^2} v_n\cdot \nabla v_n
    	\left(x, \frac{\eta(t)-t_n}{t_n - t_{n+1}}\right)\, .
    \end{align}
    Relying on \eqref{eq:non-lin est}, it follows that
    \begin{equation}
    \|v_n \cdot \nabla v_n\|_{C([0,1];C^\alpha(\T^2))}
    \le C(\alpha) 5^{-(1-\alpha)n}\, ,
    \end{equation}
    in particular,
    \begin{equation}\label{z5}
    	\|v^m \cdot \nabla v^m\|_{C([0,1];C^\alpha(\T^2))}
    	\le C(\alpha) \, ,
    \end{equation}
    and
    \begin{equation}\label{z3}
    \lim_{m\to \infty}	\| v^m \cdot \nabla v^m - v \cdot \nabla v\|_{C([0,1];C^\alpha(\T^2))}
    = 0\, .
    \end{equation}
    Let us now deal with $\nu_m \Delta v^m$. Relying on \eqref{eq:est1} we find
    \begin{equation}\label{z6}
    	\begin{split}
    	\nu_m \|\Delta v^m(\cdot,t) \|_{C^\alpha(\T^2)} 
    	& \le  \nu_m \sum_{n = 0}^m |\eta'(t)| \chi_{[t_n, t_{n+1})}(\eta(t)) \frac{1}{t_{n+1}-t_n} \norm{\Delta v_n
    	\left(\cdot , \frac{\eta(t)-t_n}{t_n - t_{n+1}}\right)}_{C^\alpha(\T^2)}
       \\& \le 
        C(\alpha) \nu_m  \sum_{n = 0}^m |\eta'(t)| \chi_{[t_n, t_{n+1})}(\eta(t)) \frac{1}{t_{n+1}-t_n}  5^{n(\alpha + 1)}
       \\& \le C(\alpha) \nu_m m^{10} 5^{m(\alpha + 1)}
       \\& \le C(\alpha) m^{20}5^{-m(1-\alpha)}\, .
       \end{split}
    \end{equation}
    The latter, together with \eqref{z4} and \eqref{z5} implies \eqref{eq: force est}. The second conclusion follows building upon \eqref{z2}, \eqref{z3} and \eqref{z6}.
 
\end{proof}

\section{Anomalous dissipation}\label{sec:anomalous diss}

Let us consider the velocity field $v^m(x,t)$ defined in \eqref{eq: vm}. We consider the problem
\begin{equation}\label{eq:adv-diff}
	\begin{dcases}
		\partial_t \theta^m + v^m \cdot \nabla \theta^m = \nu_m \Delta \theta^m \, ,
		\\
		\theta^m(x,0) = \rho(x,0) \, ,
	\end{dcases}
\end{equation}
where $\rho(x,t)$ is the scalar built in \eqref{eq: v, rho}, and $ \nu_m := m^{10}5^{-2m}$ as in \eqref{eq:nu m}.

\begin{proposition}\label{prop: anomalous diss, linear}
	Let $\theta^m$ and $v^m$ be as above. Then,
	\begin{equation}
		\liminf_{m\to 0} \nu_m \int_0^1 \int_{\T^2} |\nabla \theta^m(x,t)|^2\di x \di t > 0 \, .
	\end{equation}	
\end{proposition}
The proof of \autoref{prop: anomalous diss, linear} follows closely \cite[Proposition 1.3]{DrivasElgindiIyerJeong19}.
Let us begin by proving that $\rho(x,t)$ is concentrated in frequency around $C 5^n \sim \| \nabla \rho(\cdot,t)\|_{L^2(\T^2)}$ when $t\in [t_n, t_{n+1})$.

\begin{lemma}\label{lemma:freq control}
	There exists $\Lambda \in (0,1)$ such that for any $n\in \mathbb{N}$ and $t\in [t_n, t_{n+1})$ the following properties hold:
	\begin{itemize}
		\item[(i)] 	
		\begin{equation}
			\| \nabla \rho(\cdot ,t) \|_{L^2(\T^2)} \le \Lambda^{-1} 5^n \, ;
		\end{equation}
	    \item[(ii)] 
	        \begin{equation}
	        	\| P_{\le \Lambda 5^n} \rho(\cdot ,t)\|_{L^2(\T^2)}
	        	\le 10^{-10} \, ,
	        \end{equation}
        where $P_{\le \Lambda 5^n }$ denotes the Fourier projector on frequencies smaller than $\Lambda 5^n$.
	\end{itemize}
\end{lemma}

\begin{proof}
Fix $t\in [t_n, t_{n+1})$.
Recalling that
\begin{equation}
    	\rho(x,t) := \sum_{n\ge 0} \chi_{[t_n, t_{n+1})}(\eta(t)) 
	\rho_n \left(x, \frac{\eta(t)-t_n}{t_n - t_{n+1}}\right) \, ,
\end{equation}
we use the inequalities \eqref{e:b2-1}-\eqref{e:b2-3} in  \autoref{thm: quasi-self-sim} to get
\begin{align}
    	\| \nabla \rho(\cdot ,t) \|_{L^2(\T^2)} 
    	 \le 	\| \nabla \rho(\cdot ,t)
    	\|_{L^\infty(\T^2)}
    	 \le \norm{\nabla  \rho_n \left(\cdot , \frac{\eta(t)-t_n}{t_n - t_{n+1}}\right)  }_{L^\infty(\T^2)}
    	 \le C 5^n \, .
\end{align}
Let us prove the second conclusion. It suffices to show that
\begin{equation}
    \norm{ P_{\le \Lambda 5^n } \rho_n \left(\cdot , s \right) }_{L^2(\T^2)} \le 10^{-10}
    \quad
    \text{for every $s\in [0,1]$, $n\in \mathbb{N}$}\, ,
\end{equation}
provided $\Lambda$ is sufficiently small. We rely once more on the inequalities \eqref{e:b2-1}-\eqref{e:b2-3} in \autoref{thm: quasi-self-sim}:
\begin{equation}
    \norm{ P_{\le \Lambda 5^n } \rho_n \left(\cdot , s \right) }_{L^2(\T^2)}
    \le \Lambda 5^n \| \rho_n \left(\cdot , s \right) \|_{\dot H^{-1} (\T^2)}
    \le C \Lambda \, .
\end{equation}
Being $C$ independent on $n$ and $s\in (0,1)$, our conclusion follows.
\end{proof}

Next we state a well-known vanishing viscosity estimate for passive scalars. We refer the reader to \cite[Proposition 1.3]{DrivasElgindiIyerJeong19} for its proof.

\begin{lemma}\label{lemma:quant vanish visc}
Fix $\nu>0$, $v\in L^1([0,1];W^{1,\infty}(\T^2;\R^2)$, and $\theta_0\in L^2(\T^2)$. 
It holds
\begin{align}
   \sup_{t\le 1} \| \theta^\nu(\cdot,t) - \theta^0(\cdot,t)\|_{L^2}^2
   \le 
   \left(   2\nu \int_0^1 \| \nabla \theta^\nu(\cdot,s)\|_{L^2}^2\di s  \right)^{1/2} 
   \left( 2\nu \int_0^1 \| \nabla \theta^0(\cdot,s)\|_{L^2}^2\di s  \right)^{1/2}\, ,
\end{align}
where $\theta^\nu$ solves the advection-diffusion equation
\begin{equation}
   \begin{cases}
    \partial_t \theta^\nu + v\cdot \nabla \theta^\nu = \nu \Delta \theta^\nu \, ,
    \\
    \theta^\nu(x,0) = \theta_0(x)\, ,
   \end{cases}
\end{equation}
and $\theta^0$ solves the transport equation
\begin{equation}
   \begin{cases}
    \partial_t \theta^0 + v\cdot \nabla \theta^0 = 0 \, ,
    \\
    \theta^0(x,0) = \theta_0(x)\, .
   \end{cases}
\end{equation}
\end{lemma}

\subsection{Proof of \autoref{prop: anomalous diss, linear}}

Let us consider the auxiliary problem
\begin{equation}
	\begin{dcases}
	\partial_t \rho^m + v^m \cdot \nabla \rho^m = 0 \, ,
	\\
	\rho^m(x,0) = \rho(x,0) \, .
    \end{dcases}
\end{equation}
It turns out that
\begin{equation}\label{z7}
	\rho^m(x,t) =
	\begin{cases}
		\rho(x,t) & \text{if $t\le t_{m+1}$}
		\\
		\rho(x,t_{m+1}) & \text{if $t\in (t_{m+1}, 1]$} \, .
	\end{cases}
\end{equation}
An application of \autoref{lemma:quant vanish visc}, together with \eqref{z7}, gives
\begin{equation}\label{eq:vanishing}
	\begin{split}
	\sup_{s\le t} &
	\| \rho^m(\cdot, s) - \theta^m(\cdot, s)\|_{L^2(\T^2)}^2 
	\\&\le 
	\left(  2\nu_m \int_0^{t} \int_{\T^2} |\nabla \theta^m(x,s)|^2\di x \di s \right)^{1/2} 
	\left(  2\nu_m \int_0^{t} \int_{\T^2} |\nabla \rho^m(x,s)|^2\di x \di s \right)^{1/2}
	\\& \le 
	\left(  2\nu_m \int_0^1 \int_{\T^2} |\nabla \theta^m(x,s)|^2\di x \di s \right)^{1/2} 
	\left(  2\nu_m \int_0^{t} \int_{\T^2} |\nabla \rho(x,s)|^2\di x \di s \right)^{1/2} \, ,
    \end{split}
\end{equation}
for any $t\le t_{m+1}$.

Let us now assume by contradiction that 
\begin{equation}
	\liminf_{m\to 0} \nu_m \int_0^1 \int_{\T^2} |\nabla \theta^m(x,s)|^2\di x \di s = 0 \, .
\end{equation}
Then, for any $\delta \le 10^{-10}$ we can find an arbitrarily big $m$ such that
\begin{equation}\label{z9}
	2\nu_m \int_0^1 \int_{\T^2} |\nabla \theta^m(x,s)|^2\di x \di s \le \delta^2 \, .
\end{equation}
We claim that there exists $t^*\in (0,t_{m+1})$ such that 
\begin{equation}\label{z8}
	2\nu_m \int_0^{t^*} \int_{\T^2} |\nabla \rho(x,s)|^2\di x \di s = 1 \, . 
\end{equation}
Indeed, using the interpolation inequality $\|u\|_{L^2}^2\le \| u \|_{\dot H^1} \| u \|_{\dot H^{-1}}$, and the inequalities \eqref{e:b2-1}-\eqref{e:b2-3} in \autoref{thm: quasi-self-sim}, we have
\begin{equation}
	\begin{split}
		2\nu_m \int_0^{t_{m+1}} \int_{\T^2} |\nabla \rho(x,s)|^2 \di x \di s
		&\ge 2\nu_m \int_{t_{m}}^{t_{m+1}} \int_{\T^2} \abs{\nabla \rho_n\left(x, \frac{s-t_m}{t_{m+1}-t_m}\right)}^2 \di x \di s
		\\& \ge 2\nu_m \int_{t_m}^{t_{m+1}} \| \rho_n(\cdot,0)\|_{L^2(\T^2)}^4 \norm{ \rho_n\left(\cdot , \frac{s-t_m}{t_{m+1}-t_m}\right) }_{\dot H^{-1}(\T^2)}^{-2}
		\\& \ge  C\nu_m 5^{2m}
		\\& = Cm^{10} 
		\\& > 1 \, , \quad \text{when $m$ is big enough.}
	\end{split}
\end{equation}
From \eqref{eq:vanishing} we deduce
\begin{equation}
	\sup_{s\le t^*} 
	\| \rho(\cdot, s) - \theta^m(\cdot, s)\|_{L^2(\T^2)}^2
	=
	\sup_{s\le t^*} 
	\| \rho^m(\cdot, s) - \theta^m(\cdot, s)\|_{L^2(\T^2)}^2 
	\le \delta \, ,
\end{equation}
We show that the latter, together with \eqref{z8}, contradicts \eqref{z9}.

Let $k\le m$ such that $t_{k} < t^*$ and $t\in [t_{k}, \min\{t_{k+1}, t^*\}]$. From \autoref{lemma:freq control}(i) we deduce
\begin{equation}
	\begin{split}
		\| P_{> \Lambda 5^k }& \theta^m(\cdot,t)\|_{L^2(\T^2)}^2
		\\& = 
		\|\theta^m(\cdot,t)\|_{L^2(\T^2)}^2
		- \| P_{\le \Lambda 5^k } \theta^m(\cdot,t)\|_{L^2(\T^2)}^2
		\\& 
		\ge \|\theta^m(\cdot,t)\|_{L^2(\T^2)}^2
		- 2\| P_{\le \Lambda 5^k } (\rho(\cdot,t) -\theta^m(\cdot,t))\|_{L^2(\T^2)}^2
		-2 \|P_{\le \Lambda 5^k } \rho(\cdot,t) \|_{L^2(\T^2)}^2
		\\& \ge
         \|\theta^m(\cdot,t)\|_{L^2(\T^2)}^2 
         -2\| \rho(\cdot,t) -\theta^m(\cdot,t)\|_{L^2(\T^2)}^2 - 2\cdot 10^{-20}
         \\& \ge \|\theta^m(\cdot,t)\|_{L^2(\T^2)}^2 - 2\delta^2 - 2\cdot 10^{-20} \, .
	\end{split}
\end{equation}
Moreover, 
\begin{equation}
	\|\theta^m(\cdot,t)\|_{L^2(\T^2)}^2
	= \|\theta^m(\cdot,0)\|_{L^2(\T^2)}^2 -2\nu_m \int_0^t \int_{\T^2} |\nabla \theta^m(x, s)|^2 \di x \di s
	\ge 1 - \delta^2 \, .
\end{equation}
All in all
\begin{equation}\label{e:z}
	\| P_{> \Lambda 5^k } \theta^m(\cdot,t)\|_{L^2(\T^2)}^2
	\ge \frac{1}{2} \, ,\quad
	\text{for any $t\in [t_{k}, \min\{t_{k+1}, t^*\}]$}\, .	
\end{equation}
Let $k^*\le m$ be the biggest integer such that $t_{k^*}\le t^*$. From \eqref{e:z} and  \autoref{lemma:freq control} we conclude
\begin{equation}
	\begin{split}
		2\nu_m \int_0^{t^*} \int_{\T^2} |\nabla \theta^m (x,t)|^2 \di x \di t
		& \ge 
		2\nu_m
		\sum_{k\le k^*} \int_{t_k}^{\min\{t_{k+1}, t^*\}}\int_{\T^2} |\nabla P_{> \Lambda 5^k}\theta^m (x,t)|^2 \di x \di t
		\\& \ge 
		2\nu_m
		\sum_{k\le k^*} \int_{t_k}^{\min\{t_{k+1}, t^*\}} \frac{1}{2} \Lambda^2 5^{2k}  \di t
		\\& \ge \nu_m \Lambda^3 \sum_{k\le k^*} \int_{t_k}^{\min\{t_{k+1}, t^*\}} \int_{\T^2} |\nabla \rho(x,t)|^2 \di x \di t
		\\& \ge \nu_m \Lambda^3 \int_0^{t^*} \int_{\T^2} |\nabla \rho(x,t)|^2 \di x \di t \, ,
	\end{split}
\end{equation}
hence, from \eqref{z8} we deduce
\begin{equation}
	2\nu_m \int_0^{t^*} \int_{\T^2} |\nabla \theta^m (x,t)|^2 \di x \di t
	\ge
	\Lambda^3 \, ,
\end{equation}
which contradicts \eqref{z9} provided we choose $\delta$ small enough.

\subsection{Proof of \autoref{thm:main1 version 2}}

Let us set $\nu_m := m^{10}5^{-2m}$. We define $v^{\nu_m}:= v^m$ where the latter is defined in \eqref{eq: vm}. We let $\theta^{\nu_m} = \theta^m$ be the solution to \eqref{eq:adv-diff}. Setting $g^{\nu_m} := g^m$, from \eqref{eq:2dNS} we deduce that
 \begin{equation}
	\partial_t v^{\nu_m} + v^{\nu_m} \cdot \nabla v^{\nu_m} - \nu_m \Delta v^{\nu_m} + g^{\nu_m} \, .
\end{equation}
Moreover, \autoref{lemma: bound on g} ensures that \autoref{thm:main1 version 2}(i) is satisfied.
In order to show \autoref{thm:main2 version 2}(ii) we appeal to \autoref{prop: anomalous diss, linear}.

\section{Quasi-self-similar solutions to the forced \texorpdfstring{$2d$}{2d}-NS: second construction}

In this section we build the solution to the $2d$-NS with force that will be used in the proof of \autoref{thm:main2 version 2}. The construction is done in two steps: we first build a smooth quasi-self-similar evolution in $[0,1]^2\times \R_+$ with exponential gradient growth. The latter is used as a building block to construct a second solution to the transport equation with support concentrated on a family of cubes with controlled size. 

This construction was introduced in \cite{AlbertiCrippaMazzucato18} to provide an example of {\it instantaneous loss of regularity} for scalars advected by Sobolev regular velocity fields.
We briefly go through the entire construction, rather than taking it as a black box, because for the sake of our application we need to keep track of fine estimates on the scalar and the velocity field.

\subsection{Quasi-self-similar evolution and exponential mixing}

Let us consider $\{ (\rho_n,v_n)\, :\,  n\in \mathbb{N} \}$ as in \autoref{thm: quasi-self-sim}. We define
\begin{align}\label{eq: tilde}
    \tilde \rho(x,t) & := \sum_{n\ge 0} \chi_{[n,n+1)}(\eta(t)) \rho_n\left( x, \eta(t) - n \right)
    \\
    \tilde v(x,t) &:= \sum_{n\ge 0} \eta'(t)\chi_{[n,n+1)}(\eta(t)) v_n(x, \eta(t) - n) \, ,
\end{align}
where $\eta: \R_+ \to \R_+$ is a non-decreasing function satisfying
\begin{itemize}
    \item[(1)] $\eta(n) = n$ for any $n\in \mathbb{N}$;
    
    \item[(2)] $\frac{\di^k}{\di t^k}\eta(n) = 0$ for any $n,k\in \mathbb{N}$, $k\ge 1$;
    
    \item[(3)] $|\frac{\di^k}{\di t^k}\eta(t)| \le C(k)$ for any $t\in [0,1]$.
\end{itemize}
As noticed in \autoref{sec: first construction}, the introduction of $\eta$ is necessary to smooth out the vector field with respect to the time variable.

By relying on \autoref{thm: quasi-self-sim} it is immediate to show the following proposition, we refer the reader to \cite[Theorem 6.7]{AlbertiCrippaMazzucato16} and \cite[Theorem 6]{AlbertiCrippaMazzucato18}
for more details.

\begin{proposition}\label{prop: exp grad growth}
 The couple $(\tilde \rho, \tilde v)$ is a smooth solution to the transport equation in $[0,1]^2 \times \R_+$. Moreover, there exists a constant $C>0$ such that the following hold:
\begin{itemize}
    \item[(1)] $\tilde v(\cdot, t)$ is divergence-free, and
    \begin{equation}
        \| \tilde v(\cdot, t)\|_{L^\infty}
        +
        \| \nabla \tilde v(\cdot, t)\|_{L^\infty} \le C \, \quad
        \text{for any $t\ge 0$};
    \end{equation}
    
    \item[(2)] $\tilde\rho (\cdot, t)$ has mean zero for any $t\ge 0$, and 
    \begin{equation}
        \| \tilde \rho(\cdot, t) \|_{\dot H^{-1}}
        \le C 5^{-t}\, ,
        \quad
        \| \tilde \rho(\cdot, t) \|_{L^2} = 1\, ,
        \quad 
         \| \nabla \tilde \rho(\cdot,t) \|_{L^\infty} \le C 5^t\, ;
    \end{equation}
    
    \item[(3)] there exists a compact set $K\subset (0,1)$ such that $\supp \tilde v(\cdot, t), \supp \tilde \rho (\cdot, t) \subset K$ for any $t\ge 0$.
\end{itemize}
\end{proposition}
We can interpret the estimate
\begin{equation}\label{eq:exp-mix}
    \| \tilde \rho(\cdot, t) \|_{\dot H^{-1}} \le C 5^{-t}\, ,
    \quad \text{for every $t\ge 0$}\, ,
\end{equation}
by saying that $\tilde \rho$ undergoes an {\it exponential mixing}. A simple interpolation argument shows that \eqref{eq:exp-mix} implies exponential gradient growth, 
\begin{equation}\label{eq: exp grad growth}
\| \nabla \tilde \rho(\cdot,t)\|_{L^2} \ge 
 \frac{\| \tilde  \rho(\cdot,t)\|_{L^2}^{2}}{\|\tilde \rho(\cdot,t)\|_{\dot H^{-1}}}
 \ge 
 C^{-1}5^{ t}\, ,
\end{equation}
which saturates the bound in (2). This feature is related to the quasi-self-similar structure of the solution and will be key in what follows.


    


\subsection{Scaling and iteration}

Starting from $(\tilde \rho, \tilde v)$, we build $(\rho, v)$, a new solution to the transport equation displaying {\it anomalous loss of regularity}. The velocity field $v$ belongs to the Sobolev space $W^{1,p}$ for any $p>1$, and, when plugged into the Euler equations, produces a body force with the same spatial regularity . The latter property will be fundamental for the application to \autoref{thm:main2}.
The construction is taken from \cite{AlbertiCrippaMazzucato18}  (see also \cite{BrueNguyen18c,BrueNguyen19} for other applications of the same idea).

\medskip

For any $n\in \mathbb{N}$, we consider the following parameters
\begin{equation}\label{eq:parameters}
    \lambda_n := \frac{1}{100} e^{-n}\, ,
    \quad
    \tau_n := \frac{1}{n^3}\, ,
    \quad
    \gamma_n := e^{-n^2}\, ,
\end{equation}
and a family of disjoint cubes
\begin{equation}\label{eq:Qn}
    Q_n = (-3 \lambda_n, 3\lambda_n)^2 + x_n \subset (1/4,3/4)^2 \, .
\end{equation}
We define
\begin{align}
    \tilde v_n(x,t) : = & \frac{\lambda_n}{\tau_n} \tilde v \left( \frac{x-x_n}{\lambda_n}, \frac{t}{\tau_n}\right)\, ,
    \\
    \tilde \rho_n(x,t) : = & \gamma_n \tilde \rho \left( \frac{x-x_n}{\lambda_n}, \frac{t}{\tau_n}\right)\, .
\end{align}
and
\begin{equation}\label{eq: v rho}
    v(x,t) := \sum_n \tilde v_n(x,t) \, ,
    \quad
    \rho(x,t) := \sum_n \tilde \rho_n(x,t) \, ,
\end{equation}
Notice that $v$ and $\rho$ are compactly supported in $(0,1)^2$ uniformly in time.
With a slight abuse of notation we denote by $(\rho,v)$ their $1$-periodic extension. So, from now on the domain of definition of $v$, $\tilde v_n$, $\rho$, and $\tilde \rho_n$ will be the periodic box $\mathbb{T}^2$ of length one.

\begin{remark}\label{remark:supports}
For any $n\in \mathbb{N}$, we have
\begin{equation}
    \supp \tilde v_n(\cdot, t), \supp \tilde \rho(\cdot,t) \subset B_{\lambda_n}(x_n) \, 
    \quad 
    \text{for any $t\ge 0$}\, .
\end{equation}
In particular, if $n \neq m$ then the distance between the support of $\tilde v_n$ and $\tilde \rho_n$ is at least $\lambda_m + \lambda_n$. The same consideration holds for the densities $\tilde \rho_n$.
\end{remark}

\begin{lemma} There exists $C>0$ such that the following estimates hold uniformly in $n\in \mathbb{N}$ and $t\ge 0$:
\begin{align}
    \| \tilde \rho_n(\cdot, t) \|_{\dot H^{-1}(\mathbb{T}^2)}  &\le C \gamma_n \lambda_n^2 5^{-\frac{t}{\tau_n}}\, ,
    \label{eq: tilde rho H-1}
    \\
    \| \tilde \rho_n(\cdot, t) \|_{L^2(\T^2)} & = \gamma_n \lambda_n\, ,
    \label{eq: tilde rho L2}
    \\
     \| \nabla  \tilde \rho_n(\cdot, t)\|_{L^p(\mathbb{T}^2)}
    & \le C \gamma_n \lambda_n^{\frac{2}{p}-1} 5^{\frac{t}{\tau_n}}\, ,
    \quad
    \text{for any $p\in [1,\infty]$}\, ,
    \label{eq: tilde rho H1}
    \\
   \| \nabla v_n(\cdot, t) \|_{L^p(\T^2)} & \le C\frac{\lambda_n^{\frac{2}{p}}}{\tau_n}
    \quad 
    \text{for any $p\in [1,\infty]$}\, ,
    \label{eq: tilde v Sobolev}
\end{align}
Moreover, $\rho(\cdot, 0)\in C^\infty(\T^2)$.
\end{lemma}

\begin{proof}
   First we recall that $\tilde \rho_n$ has spatial mean zero, as a consequence of \autoref{prop: exp grad growth}. The second observation is that
   \begin{equation}
       \| \tilde \rho_n(\cdot, t) \|_{\dot H^{-1}(\mathbb{T}^2)}
       \le 10 \| \tilde \rho_n(\cdot, t) \|_{\dot H^{-1}(\R^2)}\, ,
   \end{equation}
   since the support of $\tilde \rho(\cdot, t)$ is contained in $B_{\lambda_n}(x_n) \subset (1/4,3/4)^2$ as a consequence of \autoref{remark:supports} and \eqref{eq:Qn}. We now use the standard scaling properties of the $\dot H^{-1}(\R^2)$ norm to obtain
   \begin{equation}
       \| \tilde \rho_n(\cdot, t) \|_{\dot H^{-1}(\R^2)}
       = \norm{\gamma_n \tilde \rho \left( \frac{\cdot -x_n}{\lambda_n}, \frac{t}{\tau_n}\right)}_{\dot H^{-1}(\R^2)}
       = \gamma_n \lambda_n^2 \norm{ \tilde \rho\left(\cdot, \frac{t}{\tau_n}\right) }_{\dot H^{-1}(\R^2)}\, .
   \end{equation}
   The estimate \eqref{eq: tilde rho H-1} follows from \autoref{prop: exp grad growth}(2).
   
   The check \eqref{eq: tilde rho L2}, we use standard scaling properties of the $L^2$ norm:
   \begin{equation}
       \| \tilde \rho_n(\cdot,t) \|_{L^2(\T^2)} 
       =
       \norm{\gamma_n \tilde \rho \left(\frac{\cdot - x_n}{\lambda_n}, \frac{t}{\tau_n}\right)}_{L^2(\R^2)}
       = \gamma_n \lambda_n \| \tilde \rho(\cdot, 0) \|_{L^2} = \gamma_n \lambda_n\, .
   \end{equation}
   To prove \eqref{eq: tilde rho H1}, we use the identity
   \begin{equation}\label{z11}
       D^k \tilde \rho (x,t) = \gamma_n \lambda_n^{-k} D^k \tilde \rho \left( \frac{x-x_n}{\lambda_n},\frac{t}{\tau_n}\right)
   \end{equation}
   with $k=1$, the fact that the right hand side is supported in $B_{\lambda_n}(x_n)$, and \autoref{prop: exp grad growth}(2).

   Let us pass to the proof of \eqref{eq: tilde v Sobolev}. Thanks to the identity
   \begin{equation}
       \nabla  \tilde v_n (x,t) = \frac{1}{\tau_n}  \nabla \tilde v \left( \frac{x-x_n}{\lambda_n},\frac{t}{\tau_n}\right)\, ,
   \end{equation}
   and \autoref{prop: exp grad growth}(1) we get
   \begin{equation}
       \| \tilde v_n(\cdot, t) \|_{L^p(\T^2)}
       = \frac{1}{\tau_n} \norm{\nabla \tilde v \left( \frac{x-x_n}{\lambda_n},\frac{t}{\tau_n}\right)}_{L^p(\R^2)}
       \le  C\frac{\lambda_n^{\frac{2}{p}}}{\tau_n}\, .
   \end{equation}
   
   To estimate $\| \rho(\cdot,0)\|_{C^k}$, we employ \eqref{z11}, the fact that $\tilde \rho$ is smooth, and \eqref{eq:parameters}:
   \begin{equation}
     \| \rho(\cdot,0)\|_{C^k}
     \le \sum_n
     \| \tilde \rho_n(\cdot,0)\|_{C^k}
      \le C \sum_n \gamma_n\lambda_n^{-k}
      < \infty\, .
   \end{equation}
   This completes the proof
\end{proof}

\subsection{\texorpdfstring{$2d$}{2d}-NS with body force}\label{subsec:2d-NS 2}
Let $(\rho,v)$ be as in the previous section. We smooth it out as follows: For every $m\ge 1$, define
\begin{align}
    v^m(x,t) & := \sum_{n=0}^m \tilde v_n(x,t) = \sum_{n=0}^m \frac{\lambda_n}{\tau_n} \tilde v \left( \frac{x-x_n}{\lambda_n}, \frac{t}{\tau_n}\right)\, .
    \\
    \rho^m(x,t) & := \sum_{n=0}^m
     \tilde \rho_n(x,t) 
     =
     \sum_{n=0}^m \gamma_n \tilde \rho \left( \frac{x-x_n}{\lambda_n}, \frac{t}{\tau_n}\right) \, .
\end{align}
Notice that $(\rho^m, v^m)$ is a smooth solution to the incompressible transport equation in $\mathbb{T}^2\times [0,1]$.
We then introduce the viscosity parameter
\begin{equation}\label{eq: nu m 2}
    \nu_m := \lambda_m^2 \tau_m 5^{-\frac{2}{\tau_m}} \, ,
\end{equation}
and define the body force
\begin{equation}
    g^m := \partial_t v^m + v^m \cdot \nabla v^m - \nu_m \Delta v^m \, .
\end{equation}

\begin{lemma}\label{lemma: bound on g 2}
For any $p<\infty$ there is a constant $C(p)$ such that
\begin{equation}
    \| g^m \|_{C([0,1];W^{1,p}(\T^3))} \le C(p) \, .
\end{equation}
Moreover, $g^m \to g$ in $C([0,1];W^{1,p}(\T^3))$ as $m\to \infty$, where
\begin{equation}
    g := \partial_t v + v\cdot \nabla v
\end{equation}
is the body force generated by $v$ in the Euler equations.
\end{lemma}

\begin{proof}
Notice that
\begin{align}
    g^m(x,t) = \sum_{n=0}^m \frac{\lambda_n}{\tau_n^2}( \partial_t \tilde v + \tilde v \cdot \nabla  \tilde v)\left( \frac{x-x_n}{\lambda_n}, \frac{t}{\tau_n}\right) + \frac{\nu_m}{\lambda_n \tau_n} \Delta \tilde v \left( \frac{x-x_n}{\lambda_n}, \frac{t}{\tau_n}\right)
\end{align}
Thanks to \eqref{eq: tilde}, we have the identities
\begin{align*}
(\tilde v \cdot \nabla \tilde v) \left( \frac{x-x_n}{\lambda_n}, \frac{t}{\tau_n} \right) 
= &
\sum_{k\ge 0} |\eta'\left(\frac{t}{\tau_n}\right)|^2 \chi_{[k,k+1)}\left(\eta\left(\frac{t}{\tau_n}\right)\right) (v_k \cdot \nabla v_k)\left(\frac{x-x_n}{\lambda_n},\eta\left(\frac{t}{\tau_n}\right) - k\right)\, .
\\
(\partial_t \tilde v) \left( \frac{x-x_n}{\lambda_n}, \frac{t}{\tau_n} \right)  = & 
\sum_{k\ge 0} \partial_s \left(\eta'\left(s\right)\chi_{[k,k+1)}\left(\eta\left(s\right)\right)\right)|_{s=\frac{t}{\tau_n}}\,  v_k\left(\frac{x-x_n}{\lambda_n},\eta\left(\frac{t}{\tau_n}\right) - k\right) 
\\ &  + \sum_{k\ge 0} \eta'\left(\frac{t}{\tau_n}\right)\chi_{[k,k+1)}\left(\eta\left(\frac{t}{\tau_n}\right)\right)\eta'\left(\frac{t}{\tau_n}\right)(\partial_t v_k) \left(\frac{x-x_n}{\lambda_n},\eta\left(\frac{t}{\tau_n}\right) - k\right)\, .
\end{align*}
We use \eqref{eq:non-lin est}, and the fact that $(v_k \cdot \nabla v_k)\left(\frac{x-x_n}{\lambda_n},s\right)$ is supported in $B_{\lambda_n}(x_n)$ for every $s\ge 0$, to get
\begin{equation}
    \norm{(\tilde v \cdot \nabla \tilde v) \left( \frac{\cdot -x_n}{\lambda_n}, \frac{t}{\tau_n} \right)  }_{W^{1,p}}^p
    \le C(p) \sum_{k\ge 0} \chi_{[k,k+1)}(\eta(t/\tau_n)) \lambda_n^{-p} \lambda_n^2 
    = C(p) \lambda_n^{2-p}\, ,
\end{equation}
while \autoref{thm: quasi-self-sim}(a) gives
\begin{equation}\label{z13}
    \norm{(\partial_t \tilde v) \left( \frac{\cdot - x_n}{\lambda_n}, \frac{t}{\tau_n} \right)}_{W^{1,p}}^p
    \le 
    C(p) \sum_{k\ge 0} \chi_{[k,k+1)}(\eta(t/\tau_n)) \lambda_n^{-p} \lambda_n^2 
    = C(p) \lambda_n^{2-p}\, .
\end{equation}
We employ \eqref{eq: tilde} once more to write
\begin{equation}\label{z14}
     \Delta \tilde v \left( \frac{x-x_n}{\lambda_n}, \frac{t}{\tau_n}\right)
     =
     \sum_{k\ge 0} \eta'(t/\tau_n) \chi_{[k,k+1)}(\eta(t/\tau_n)) (\Delta v_k)\left( \frac{x-x_n}{\lambda_n}, \eta(t/\tau_n) - k\right)\, ,
\end{equation}
hence, \autoref{thm: quasi-self-sim}(a) gives
\begin{equation}\label{z15}
   \norm{  \Delta \tilde v \left( \frac{\cdot - x_n}{\lambda_n}, \frac{t}{\tau_n}\right)}_{W^{1,p}}^p 
   \le 
   C(p)\sum_{k\ge 0}\chi_{[k,k+1)}(\eta(t/\tau_n)) 5^{2k p} \lambda_n^{2-p}
   \le C(p) 5^{\frac{2t}{\tau_n}p} \lambda_n^{2-p}\, .
\end{equation}
Hence, collecting \eqref{z13}, \eqref{z14}, \eqref{z15}, and using \eqref{eq: nu m 2}, \eqref{eq:parameters} we get 
\begin{align*}
    \| g^m(\cdot, t)\|_{W^{1,p}}^p
    & \le 
    C(p) \sum_{n=0}^m \left( \frac{\lambda_n^2}{\tau_n^{2p}}  + \lambda_n^2\left(\frac{\nu_m}{\lambda_n^2 \tau_n} 5^{2\frac{t}{\tau_n}}\right)^p \right)
    \\& \le 
    C(p)\sum_{n=0}^m  \frac{\lambda_n^2}{\tau_n^{2p}} +
    C(p) \sum_{n=0}^m \lambda_n^2 \le
    C(p) \, 
\end{align*}
for any $t\in [0,1]$.

The convergence $g^m \to g$ in $C([0,1];W^{1,p}(\T^2))$ follows from \eqref{eq:parameters} and
\begin{equation}
    \| g(\cdot, t) - g^m(\cdot, t) \|_{W^{1,p}} \le C(p) \sum_{n>m} \frac{\lambda_n^2}{\tau_n^{2p}} + \lambda_n^2\, ,
    \quad
    \text{for any $t\in [0,1]$}\, .
\end{equation}
This completes the proof.
\end{proof}

\begin{remark}\label{eq: reg vm 2}
It turns out that the velocity field $v^m$ enjoys the same regularity of $g^m$. It is indeed immediate to see from \eqref{eq: tilde v Sobolev} that $v^m\in C([0,1]; W^{1,p})$, uniformly in $m$.
\end{remark}

\section{Enhanced dissipation}

Let $(\rho^m, v^m)$, $g^m$, and $\nu_m$ as in \autoref{subsec:2d-NS 2}.
We define $\theta^m$ by solving
\begin{equation}\label{eq:adv-diff 2}
	\begin{dcases}
		\partial_t \theta^m + v^m \cdot \nabla \theta^m = \nu_m \Delta \theta^m \, ,
		\\
		\theta^m(x,0) = \rho(x,0) \, ,
	\end{dcases}
\end{equation}
where $\rho(\cdot, t)$ was defined in \eqref{eq: v rho}.

\begin{proposition}\label{prop:key2}
 There exists $C>0$ such that
 \begin{equation}
     \nu_m \int_0^1 \int_{\T^2} |\nabla \theta^m(x,t)|^2 \di x \di t \ge C \gamma^4_m\lambda^4_m \tau_m^2\, ,
 \end{equation}
 for any $m\ge 3$.
\end{proposition}
Let us begin by proving that $\rho^m(\cdot,t)$ has a non-trivial amount of frequencies bigger than $\gamma_m 5^{\frac{t}{\tau_m}}$, while its Lipschitz norm does not exceed $5^{\frac{t}{\tau_m}}$. Once again, this is related to the quasi-self-similar structure of the solution.

\begin{lemma}\label{lemma:freq control 2}
There exists $\kappa > 0 $ such that the following hold for every $m\ge 3$ and $t\in[0,1]$.
\begin{itemize}
    \item[(1)] Set $\Lambda:=\kappa \gamma_m 5^{\frac{t}{\tau_m}}$,  we have
    \begin{equation}
        \| P_{> \Lambda} \rho^m(\cdot, t)\|_{L^2(\mathbb{T}^2)} \ge \frac{1}{2}\gamma_m\lambda_m \ ,
    \end{equation}
    where $P_{>\Lambda}$ is the Fourier projectors on frequencies bigger than $\Lambda$;
    
    \item[(2)] $\| \nabla \rho^m(\cdot,t)\|_{L^\infty(\T^2)}\le C 5^{\frac{t}{\tau_m}}$.
\end{itemize}
\end{lemma}

\begin{proof}
   
   
 Let us prove (1). We estimate
   \begin{align}
       \| P_{\le \Lambda} \rho^m(\cdot, t) \|_{L^2(\mathbb{T}^2)} 
       & \le 
       \| P_{\le \Lambda} \sum_{n=0}^{m-1} \tilde \rho_n(\cdot, t) \|_{L^2(\mathbb{T}^2)}
       +
       \| P_{\le \Lambda} \tilde \rho_m(\cdot, t) \|_{L^2(\mathbb{T}^2)}
       \\& \le
       \| \sum_{n=0}^{m-1} \tilde \rho_n(\cdot, t) \|_{L^2(\mathbb{T}^2)}
       +
       \| P_{\le \Lambda} \tilde \rho_m(\cdot, t) \|_{L^2(\mathbb{T}^2)}\, .
   \end{align}
   Thanks to \eqref{eq: tilde rho L2}, we get
   \begin{equation}
       \| \rho^m(\cdot,t) \|^2_{L^2(\mathbb{T}^2)}
       = \sum_{n=0}^m \gamma_n^2 \lambda_n^2
       =  \| \sum_{n=0}^{m-1} \tilde \rho_n(\cdot, t) \|_{L^2(\mathbb{T}^2)}^2 + \gamma_m^2\lambda_m^2\, ,
   \end{equation}
   while \eqref{eq: tilde rho H-1} gives
   \begin{equation}
        \| P_{\le \Lambda} \tilde \rho_m(\cdot, t) \|_{L^2(\mathbb{T}^2)}
        \le \Lambda \| \tilde \rho_m(\cdot, t) \|_{\dot H^{-1}(\mathbb{T}^2)}
        \le C \Lambda \gamma_m \lambda_m^2 5^{-\frac{t}{\tau_m}}
        \le \frac{1}{10}\gamma_m^2 \lambda_m^2\, ,
   \end{equation}
   up to choosing $\kappa$ small enough.
   Collecting all the estimates above we deduce
   \begin{equation}
       \| P_{\le \Lambda} \rho^m(\cdot, t) \|_{L^2(\mathbb{T}^2)}
       \le \sqrt{ \|\rho^m(\cdot, t) \|_{L^2(\mathbb{T}^2)}^2 - \gamma_m^2\lambda^2_m } +  \frac{1}{10}\gamma_m^2 \lambda_m^2 \, .
   \end{equation}
   By recalling that $\|\rho^m(\cdot,t) \|^2_{L^2(\mathbb{T}^2)} = \sum_{n=0}^m \gamma_n^2 \lambda_n^2\le 2$, and \eqref{eq:parameters}, we get
   \begin{align*}
       \| P_{\le \Lambda} &\rho^m(\cdot, t) \|_{L^2(\mathbb{T}^2)}^2
       \\& 
       \le \|\rho^m(\cdot, t) \|_{L^2(\mathbb{T}^2)}^2 - \gamma_m^2\lambda_m^2 + \frac{1}{100}\gamma_m^4 \lambda_m^4 + \frac{1}{5}\gamma_m^2 \lambda_m^2\sqrt{ \|\rho^m(\cdot, t) \|_{L^2(\mathbb{T}^2)}^2 - \gamma_m^2\lambda_m^2 }
       \\& \le   \|\rho^m(\cdot, t) \|_{L^2(\mathbb{T}^2)}^2 - \frac{1}{2}\gamma_m^2\lambda_m^2\, .
   \end{align*}
   Let us now prove (2). We employ \eqref{eq: tilde rho H1} and \eqref{eq:parameters} to get
   \begin{equation*}
       \| \nabla \rho^m(\cdot, t) \|_{L^\infty(\T^2)}
       \le \sum_{n=0}^m \| \nabla \tilde \rho_n(\cdot,t) \|_{L^\infty(\T^2)}
       \le C \sum_{n=0}^m \gamma_n \lambda^{-1}_n 5^{\frac{t}{\tau_n}}
       \le C  5^{\frac{t}{\tau_m}}\, . \qedhere
   \end{equation*}
\end{proof}

\subsection{Proof of \autoref{prop:key2}}

Recall that
\begin{equation}
	\begin{dcases}
		\partial_t \rho^m + v^m \cdot \nabla \rho^m = 0 \, ,
		\\
		\rho^m(x,0) = \rho(x,0) \, ,
	\end{dcases}
\end{equation}
hence, from \autoref{lemma:quant vanish visc} we deduce
\begin{align}
   \sup_{t\in [0,1]} \| \theta^m(\cdot, t) - & \rho^m(\cdot,t) \|_{L^2(\mathbb{T}^2)}^2
    \\& \le \left( 2\nu_m \int_0^1 \| \nabla \rho^m(\cdot,s)\|^2_{L^2(\mathbb{T}^2)}\di s \right)^{1/2}
    \left( 2\nu_m \int_0^1 \| \nabla \theta^m(\cdot,s)\|^2_{L^2(\mathbb{T}^2)}\di s \right)^{1/2}
    \\&
    \le  C\tau_m^{1/2} \lambda_m
    \left( 2\nu_m \int_0^1 \| \nabla \theta^m(\cdot,s)\|^2_{L^2(\mathbb{T}^2)}\di s \right)^{1/2}\, ,
\end{align}
where in the second inequality we used \eqref{eq: nu m 2} and \autoref{lemma:freq control 2}.

If we let $\Lambda(t):=\kappa \gamma_m 5^{\frac{t}{\tau_m}}$ be as in \autoref{lemma:freq control 2}, we have
\begin{align*}
    \| \theta^m(\cdot, t) -  \rho^m(\cdot,t) \|_{L^2(\mathbb{T}^2)}
    & \ge
     \| P_{>\Lambda(t)}(\theta^m(\cdot, t) - \rho^m(\cdot,t)) \|_{L^2(\mathbb{T}^2)}
     \\& \ge 
     \| P_{>\Lambda(t)}\rho^m(\cdot, t)  \|_{L^2(\mathbb{T}^2)} - \| P_{>\Lambda(t)}\theta^m(\cdot, t)  \|_{L^2(\mathbb{T}^2)}
     \\& \ge 
     \frac{\gamma_m\lambda_m}{2}  - \frac{1}{\Lambda(t)}  \| \nabla \theta^m(\cdot,t)\|_{L^2(\T^2)}\, \, .
\end{align*}
Hence, for any $t\in [0,1]$ we conclude
\begin{equation}\label{z10}
    \frac{\gamma_m^2 \lambda_m^2 \Lambda(t)^2}{4} \le  2 \| \nabla \theta^m(\cdot,t)\|_{L^2(\T^2)}^2 +  C\tau_m^{1/2} \lambda_m \Lambda(t)^2
    \left( 2\nu_m \int_0^1 \| \nabla \theta^m(\cdot,s)\|^2_{L^2(\mathbb{T}^2)}\di s \right)^{1/2}\, .
\end{equation}
We integrate \eqref{z10} with respect to $t\in (0,1)$, and use the identities
\begin{equation}
    \nu_m=\lambda_m^2 \tau_m 5^{-\frac{2}{\tau_m}}\, ,
    \quad
    \int_0^1 \Lambda(t)^2\di t = C \gamma_m^2 \tau_m (5^{\frac{2}{\tau_m}}-1)\, ,
\end{equation}
to get
\begin{equation*}
    C \gamma_m^4 \lambda_m^4 \tau_m^2
    \le 
    2\nu_m\int_0^1 \| \nabla \theta^m(\cdot,s)\|_{L^2(\T^2)}^2\di s + 
    \gamma_m^2 \lambda_m^4 \tau_m^{5/2}\left( 2\nu_m \int_0^1 \| \nabla \theta^m(\cdot,s)\|^2_{L^2(\mathbb{T}^2)}\di s \right)^{1/2}\, \, ,
\end{equation*}
from which the desired conclusion follows easily.

\subsection{Proof of \autoref{thm:main2 version 2}}

Let us set $\nu_m := \lambda_m^2 \tau_m 5^{-\frac{2}{\tau_m}}$. We define $v^{\nu_m}:= v^m$ and $g^{\nu_m}:= g^m$ as in \autoref{subsec:2d-NS 2}. We let $\theta^{\nu_m} = \theta^m$ be the solution to \eqref{eq:adv-diff 2}. We deduce that
 \begin{equation}
	\partial_t v^{\nu_m} + v^{\nu_m} \cdot \nabla v^{\nu_m} - \nu_m \Delta v^{\nu_m} + g^{\nu_m} \, .
\end{equation}
Moreover, \autoref{lemma: bound on g 2} ensures that \autoref{thm:main2 version 2}(i) is satisfied.
In order to show \autoref{thm:main1 version 2} we employ \autoref{prop:key2} and \eqref{eq:parameters}, obtaining
 \begin{equation}\label{eq:enhanced diss}
     \nu_m \int_0^1 \int_{\T^2} |\nabla \theta^m(x,t)|^2 \di x \di t
     \ge C e^{-8 m^2}\, ,
     \qquad
     \nu_m \le e^{-2m^3}
     \, ,
 \end{equation}
which clearly implies \autoref{thm:main2 version 2}(ii).

\appendix

\section{Converge of Navier-Stokes to classical solutions of Euler}\label{Appendix}

In this section, we show simple computation that allows to compare classical (Lipschitz solutions) of forced Euler to classical solutions of Navier-Stokes. The computation can be generalized to weaker notion of solutions of Navier-Stokes, namely distributional solutions that satisfy a suitable form of the global energy inequality, but we leave the technical details to the reader.

\begin{lemma}\label{l:computation}
Let $u^\nu$ be a classical solution of \eqref{NS} and $u$ be a Lipschitz solution of 
\begin{equation}
\left\{\begin{array}{l}
\partial_t u + u\cdot \nabla u + \nabla p = f\\
\div u = 0\, 
\end{array}
\right.
\end{equation}
on the time interval $[0,T]$.
Then the following inequality is valid for every $t\in [0,T]$: 
\begin{align}
 &\frac{\di}{\di t} \frac{1}{2} \int |u - u^\nu|^2 (x,t)\, \di x\nonumber\\
\leq & \left(\|\nabla u\|_{L^\infty}+\frac{1}{2}\right) \int |u - u^\nu|^2 (x,t)\, \di x
+ \frac{1}{2} \int |f-f^\nu|^2 (x,t)\, \di x + \frac{\nu}{4} \int |\nabla u|^2 (x,t)\, \di x\, .\label{e:Gronwall}
\end{align}
\end{lemma}

Clearly, from \eqref{e:Gronwall} one concludes immediately that, if $\nu\downarrow 0$, $f^\nu \to f$ in $L^1([0,T];L^2)$, and $u^\nu (\cdot, 0) \to u (\cdot, 0)$ in $L^2$, then $u^\nu \to u$ in $C ([0,T]; L^2)$.

\begin{proof}
We first subtract the two equations to get 
\[
\partial_t (u^\nu - u) = \underbrace{(- u^\nu \cdot \nabla u^\nu + u\cdot \nabla u)}_{=:R} - \nabla (p^\nu - p) + \nu \Delta u^\nu + (f^\nu - f)\, .
\]
We then multiply by $u^\nu -u$, integrate in space, and use the fact that $u^\nu -u$ is divergence free to conclude
\begin{align*}
& \frac{\di}{\di t} \frac{1}{2}  \int |u - u^\nu|^2 (x,t)\, \di x\\
= & \int [R \cdot (u^\nu - u)] (x,t)\, \di x + \int (f^\nu - f)\cdot (u^\nu - u) (x,t)\, \di x\,
- \nu \int [ \nabla u^\nu : (\nabla u^\nu - \nabla u) ] (x,t)\, \di x  ,
\end{align*}
where $A:B$ is the Hilbert-Schmidt product. For the first integrand in the last line the following is a very well-known fact which can be proved by elementary integration by parts:
\[
\int [R \cdot (u^\nu - u)] (x,t)\, \di x \leq 
\|\nabla u\|_{L^\infty} \int |u^\nu - u|^2 (x,t)\, \di x\, .
\]
The second integrand can be bounded in an elementary way by
\[
\int (f^\nu - f)\cdot (u^\nu - u) (x,t)\, \di x \leq 
\frac{1}{2} \int |f^\nu -f|^2 (x,t)\, \di x + \frac{1}{2} \int |u^\nu - u|^2 (x,t)\, \di x\, .
\]
As for the third integrand we can estimate it as
\begin{align*}
&- \nu \int [ \nabla u^\nu : (\nabla u^\nu - \nabla u) ] (x,t)\, \di x\\
= & - \nu \int |\nabla (u^\nu - u)|^2 (x,t)\, \di x - 
\nu \int [ \nabla u : \nabla (u^\nu - u)] (x,t)\, \di x
\leq  \frac{\nu}{4} \int |\nabla u|^2 (x,t)\, \di x\, .
\end{align*}
Combining these facts together we easily conclude \eqref{e:Gronwall}.
\end{proof}

\end{document}